\documentclass{amsart}

\usepackage{amsfonts}
\usepackage{amsmath}
\usepackage{enumerate}
\usepackage{mathrsfs}
\usepackage[textsize=small]{todonotes}

\newtheorem{theorem}{Theorem}

\newtheorem{corollary}[theorem]{Corollary}

\newtheorem{lemma}[theorem]{Lemma}

\newtheorem{remark}[theorem]{Remark}

\newcommand{\E}{\mathbf{E}}
\newcommand{\Pp}{\mathbf{P}}
\newcommand{\R}{\mathbb{R}}
\newcommand{\N}{\mathbb{N}}

\newcommand{\ONE}{{\bf 1}}

\newcommand{\Fc}{\mathscr{F}}
\newcommand{\xcmd}[2]{\left[#1\,\left\vert\vphantom{#1#2}\right. #2\right]}

\newcommand{\xbr}[1]{\left\{#1\right\}}

\newcommand{\parbar}[1]{\left( #1 \right)}
\newcommand{\Hspace}[2]{\mathscr{H}^{#1}_S ( #2 ) }

\newcommand{\be}{\begin{equation}}
\newcommand{\ee}{\end{equation}}
\newcommand{\beq}{\begin{eqnarray*}}
\newcommand{\eeq}{\end{eqnarray*}}

\newcommand{\beqn}{\begin{eqnarray}}
\newcommand{\eeqn}{\end{eqnarray}}

\newcommand{\ba}{\begin{aligned}}
\newcommand{\ea}{\end{aligned}}
\newcommand{\bes}{\begin{equation*}}
\newcommand{\ees}{\end{equation*}}
\newcommand{\Om}{\Omega}
\newcommand{\clf}{\mathcal{F}}
\newcommand{\clg}{\mathcal{G}}
\newcommand{\cln}{\mathcal{N}}
\newcommand{\Hh}{\mathbf{H}}
\newcommand{\cla}{\mathcal{A}}
\newcommand{\clh}{\mathcal{H}}
\newcommand{\zz}{\zeta}

\begin{document}

\title[Forward-Backward SDE and KPZ Equation] 
{Infinite Dimensional Forward-Backward Stochastic Differential Equations and
the KPZ Equation.}

%
\author{Sergio A. Almada Monter}
\address[Sergio A. Almada Monter and Amarjit Budhiraja]{Department of Statistics and Operations Research\\
       University of North Carolina\\
        Chapel Hill, NC 27599}
\email[Sergio A. Almada]{salmada3@email.unc.edu}
\author{Amarjit Budhiraja}
\email[Amarjit Budhiraja]{budhiraj@email.unc.edu}
\thanks{Research supported in part by  the National Science Foundation (DMS-1004418, DMS-1016441),
the Army Research Office (W911NF-10-1-0158) and the US-Israel
Binational Science Foundation (Grant 2008466).}

\date{July 20, 2012}


\maketitle

\begin{abstract}
Kardar-Parisi-Zhang (KPZ) equation is a quasilinear stochastic partial differential equation(SPDE) driven by a space-time white noise.
In recent years there have been several works directed towards giving a rigorous meaning to a solution of this equation.
Bertini, Cancrini and Giacomin \cite{BertiniCancrini, BertiniGiacomin} have proposed a notion of a solution through a limiting procedure and a certain renormalization of the nonlinearity.  In this work we study connections between the KPZ equation and certain infinite dimensional
forward-backward stochastic differential equations.  Forward-backward equations with a finite dimensional noise have been studied extensively, mainly motivated by problems in mathematical finance.  Equations considered here differ from the classical works in that, in addition to having an infinite dimensional
driving noise, the associated SPDE involves a non-Lipschitz (namely a quadratic) function of the gradient.  Existence and uniqueness of solutions of
such infinite dimensional forward-backward equations is established and the terminal values of the solutions are then used to give a new probabilistic 
representation for the solution of the KPZ equation. \\

\end{abstract}

\section{Introduction}
In \cite{PengPardoux} probabilistic representations for solutions of
certain quasilinear stochastic partial differential equations(SPDE) in terms of
finite dimensional forward-backward  stochastic differential
equations have been studied. The driving noise in the SPDE of~\cite{PengPardoux} is a finite dimensional Brownian motion.
  The paper shows that such  representations can be used to prove
existence and uniqueness results for the associated quasilinear
equation.  In recent years there have been many works that have established similar probabilistic representations for more general partial differential equations;
 see for example ~\cite{Nizar2BSDE} for such results for fully nonlinear equations.  
 Backward stochastic differential equations have a long history of applications  in financial mathematics; see~\cite{JinBook}  for  a survey of the field;  also see~\cite{MaInsurance} or~\cite{PowerPlants} for some modern applications.   They have also motivated numerical methods for nonlinear partial differential equations; see for example~\cite{TouziNumerical} or~\cite{DynamicProgramming}.

In this work we are interested in a
quasilinear SPDE driven by a space-time white noise of the following form
\be \label{eq:kpz}
\partial_t h(t,x) = \partial^2_x h(t,x) /2 - ( \partial_x h(t,x) )^2 /2 +
\mathscr{W}_t(x), \; t \in [0, T],\; x \in \R .
\ee
Here $\mathscr{W}$ is the formal white noise field:
\[
\E \mathscr{W}_t (x) \mathscr{W}_s (y) = \delta (t-s) \delta(x-y),
\]
where $\delta$ is the Dirac delta function.  The above equation,
known as the Kardar-Parisi-Zhang (KPZ) equation has been proposed in
\cite{KPZ} to describe the long scale behavior of interface
fluctuations in certain random polymer growth models.  The solution
$h(t,x) \equiv h_t(x)$ represents the height of the interface at
time $t$ and location $x$. The equation \eqref{eq:kpz} is ill posed in that, due to lack of spatial regularity of the noise, differentiable solutions do not exist and the nonlinear term on the right side of the
equation does not allow a weak sense  formulation of a solution.  In \cite{BertiniCancrini} and \cite{BertiniGiacomin} an interpretation of a solution of \eqref{eq:kpz} is proposed through a limiting procedure and a certain
`Wick renormalization' of the nonlinear term.  The Bertini-Cancrini-Giacomin (BCG) solution of \eqref{eq:kpz} was shown in \cite{BertiniGiacomin} to arise as a scaling limit of the fluctuation field for a microscopic interface model
known as the `weakly asymmetric single step solid on solid process'. 
In recent years there have been several interesting papers that have studied scaling limits for similar particle models characterizing them through  equations of the form 
\eqref{eq:kpz}; see for example ~\cite{Corwin1+1},~\cite{Jara}, ~\cite{CorwinSurvey} and references therein. In a different direction, a recent paper ~\cite{Hairer} has  proposed a notion of a solution of~\eqref{eq:kpz} using rough path theory.

In this work we give a probabilistic representation of the BCG solution of \eqref{eq:kpz} through solutions of certain infinite dimensional forward-backward  stochastic differential equations.
These equations extend the models studied in \cite{PengPardoux} to a setting where the driving noise is infinite dimensional.  Another significant difference from the setting considered in \cite{PengPardoux} is that the equation \eqref{eq:kpz}
involves a quadratic function of the gradient while \cite{PengPardoux} considers the case of a Lipschitz non-linearity. Quadratic non-linearity
has been studied by several authors in the setting of finite dimensional backward stochastic differential equations; see for example  \cite{Kobylanski},~\cite{Tevzadze}. 
However, none of these works treat equations involving both forward and backward stochastic integrals or the setting of an infinite dimensional noise.

A precise description of the representation obtained in this work requires some mathematical notation and background, which is given in Section \ref{two}, however below we proceed formally in order to describe the basic idea. Denote by $G_t$ the standard heat kernel on $\R \times \R$.  Then a solution of \eqref{eq:kpz} can formally be expressed as
\begin{equation} \label{eqn: Intuitive}
h(S,x) = G_S \star h_0 (x) + \int_0^S G_{S-r} \star \parbar{ \partial_x h (r,\cdot)  }^2 (x) dr + \int_{\R \times [0,S]} G_{S-r}(x,y)   \mathscr{W} ( dy , dr),
\end{equation}
for $S \in [0, T]$ and $x \in \R$, and where $\star$ denotes the  convolution in space and $h_0(x) = h(0,x)$ is the initial condition for \eqref{eq:kpz}. Fix  $S \in [0, T]$ and $x \in \R$. Let $W$ be a standard Brownian Motion independent of $\mathscr{W}$  and $X_r^S(x)=x + W(S) - W(r)$, for $r \in [0,S]$ . Then one can rewrite the expression in the above display as
\begin{equation} \label{eqn: Intuitive2}
h(S,x)= \E \xcmd{ h_0 (X_0^S (x) ) + \int_0^S  \parbar{ \partial_x h (r, X_r^S(x))  }^2 dr +  \int_0^S  \mathscr{W} ( X_r^S(x) , dr)  }{\Fc^\mathscr{W}},
\end{equation} 
where $\Fc^\mathscr{W}$ denotes the $\sigma$-field generated by $\mathscr{W}$.  The stochastic integral on the right side above is of course entirely formal (as is much of this description).
Define stochastic processes 
\[
z_S (r,x) = \partial_x h (r, X_r^S (x)  ), \; y_S (r,x) = h(r, X_r^S (x) ),\; r \in [0, S].
\]
Note that the values of these processes at time  $r$  depend on the past values (i.e. values over $[0,r]$) of $\mathscr{W}$ and the future increments (those over $[r,S]$)
of $W$. Also note that $y_S(S,x) = h(S, x)$.
Let, for $0 \le r \le S$, $\Fc^{\mathscr{W}}_r \vee \Fc^W_{r,S}$ be the $\sigma$ -field generated by $\{\mathscr{W}(s,x), s\le r, x \in \R\}$
and $\{W(S)-W(u), 0 \le r \le u \le S\}$.  Then~\eqref{eqn: Intuitive2} can be written as 
\begin{align*}
h(S,x)=& y_S(S,x)\\
 =& \E \xcmd{  h_0 (X_0^S (x) ) + \int_0^S z_S (r,x)^2 dr +  \int_0^S  \mathscr{W} ( X_r^S(x) , dr)  }{\Fc^{\mathscr{W}}_S \vee \Fc^W_{S,S} }.
\end{align*}
The above formula suggests an evolution equation for $y_S(u,x)$, $0 \le u \le S$ of the following form
\begin{equation}\label{eq:1020}
y_S(u,x) =  h_0 (X_0^S (x) ) + \int_0^u z_S (r,x)^2 dr +  \int_0^u \mathscr{W} ( X_r^S(x) , dr)  + M_S(u,x), \; 0 \le u \le S, 
\end{equation}
where the process  $M_S (u,x)$, $u \in [0,S]$, is such that  
\[
\E \xcmd{ M_S (u,x) } {\Fc^{\mathscr{W}}_u \vee \Fc^W_{u,S}}=0,
\] 
Thus if one can make \eqref{eq:1020} rigorous, one can then obtain a solution $h(S,x)$ of \eqref{eq:kpz} by evaluating  the solution
of \eqref{eq:1020}  at $u=S$.  The goal of this work is to show that after a suitable mollification of the infinite dimensional noise, the above equation
can indeed be interpreted in a rigorous manner; that there is a unique pair of processes $(z_S(\cdot,x), y_S(\cdot,x))$ (in a suitable class) that satisfy the equation; and that
the BGC solution of \eqref{eq:kpz} can be obtained by evaluating the solution of this `mollified equation' at $u=S$ and taking limit of the parameter of
mollification.

In the next section, we give a precise formulation and present our main results.

\section{Mathematical Preliminaries and Main Results.}  
\label{two}

In order to state our precise representation for the BCG solution of
\eqref{eq:kpz} we need some notation.  Let $\Hh = L^2(\mathbb{R},
dx)$, i.e. the Hilbert space of square integrable (with respect to
the Lebesgue measure) functions on the real line.  We will denote
the inner product and the norm on $\Hh$ by $\langle \cdot, \cdot
\rangle$ and $\| \cdot \|$, respectively. Let $(\Om, \clf, \Pp)$ be
a complete probability space on which is given a collection of
continuous real stochastic processes $\{B_t(h); t \ge 0\}_{h \in
\Hh}$ that defines a cylindrical Brownian motion (c.B.m) on $\Hh$.
Namely,
\begin{itemize}
\item $B_t(0) = 0$ and for each nonzero $h \in \Hh$,  $B_t(h) \langle h, h\rangle^{-1/2}$ is a one dimensional standard Wiener process.
\item For each $h \in \Hh$, $\{B_t(h)\}_{t\ge 0}$ is a $\clf_t^B$ martingale, where $\clf_t^B = \sigma \{B_s(v): s \le t, v \in \Hh\}\vee
\cln$ and $\cln$ is the collection of all $\Pp$ null sets.
\end{itemize}
Next, following \cite{BertiniGiacomin}, we consider a regularized
version of \eqref{eq:kpz}.  Let $\zz \in C_0^{\infty}(\mathbb{R})$ [
space of smooth functions on the real line with compact support] be
a nonnegative even function such that $\int_{\R} \zz(x) dx = 1$. For
$k \in \N$, let $\zz^k(y) = k \zz(ky)$, $y \in \R$.  For $x \in \R$,
define $\zz^k_x \in C_0^{\infty}(\mathbb{R})$  as $\zz^k_x(y) =
\zz^k(x-y)$, $y \in \R$.  Consider the Gaussian random field
$$B^k(t,x) = B_t(\zz^k_x), \; t \ge 0, x \in \R $$
with covariance
\[
\E B^k(t,x) B^{k^\prime}(s,y)= ( t \wedge s ) C^k ( x-y), \quad
x,y \in\R ,\; t,s\in[0,\infty),
\] where
$$C^k (x)= \zz^k \star \zz^k (x) \equiv \int_{\R} \zz^k_x(y) \zz^k(y)
dy.$$
 Note that $C^k(0) = k \|\zz\|^2$.
 
 The mollified KPZ equation (see \cite{BertiniGiacomin}) is given as follows.
\be \label{eq:mollkpz}
h^k(t,x) = h_0(x)  +   \frac{1}{2} \int_0^t  \left( \partial^2_x h^k(s, x) -  \left( ( \partial_x h^k  (s,x) )^2 - C^k(0) \right) \right)ds  + B^k(t,x).
\ee
The  initial condition $h_0$ is a $C(\R)$ valued random variable, independent of $B$, satisfying the following integrability condition
\begin{equation} \label{eqn: Condition_Initial}
\mbox{  for every } p>0 \mbox{ there exist } a_p >0 \mbox{ such that } \sup_{ x \in \R } e^{ - a_p |x| } \E e^{ p |h_0 (x)| } \equiv b_p < \infty.
\end{equation}
The hypothesis imposed on the initial condition in (2.4) of~\cite{BertiniGiacomin} is  weaker than the integrability condition in \eqref{eqn: Condition_Initial}, 
but the condition we impose covers all classical cases, in particular the combinations of the so called Brownian and Flat geometries (see last column of Table 5 in ~\cite{CorwinSurvey}). Note also that the assumption  in \eqref{eqn: Condition_Initial} (and also condition  (2.4) of~\cite{BertiniGiacomin}) exclude
the model setting considered in   ~\cite{Corwin1+1}, and~\cite{Spohn} where the initial condition is a distribution.

Solution of \eqref{eq:mollkpz} over any fixed time interval $[0,T]$ is understood in the weak sense, namely, it is a $\{\clf_t^B\}$- adapted stochastic process
$\{h^k(t,\cdot)\}_{0 \le t \le T} \equiv \{h^k_t\}_{0 \le t \le T}$
with sample paths in
$C([0,T]: C(\R)) \cap C((0,T]: C^1(\R))$, such that for every smooth function $\varphi$ on $\R$ with a compact support
$$
h^k_t(\varphi) = h_0(\varphi)  +   \frac{1}{2} \int_0^t  \left[ h^k_s(\varphi'') -  \left( ( \partial_x h^k_s )^2 - C^k(0) \right)(\varphi) \right]ds  + B^k_t(\varphi)$$
where for a  function $g$ on $\R$ (with suitable integrability properties), and $\varphi \in \mathcal{D}(\R)$, $g(\varphi) = \int_{\R} g(x) \varphi(x) dx$.
Here $C(\R)$ [resp. $C^1(\R)$] is the space of continuous [resp. continuously differentiable] functions on the real line.

The paper \cite{BertiniGiacomin} shows that there is a unique solution of \eqref{eq:mollkpz} in the class of  processes that satisfy
$$
\sup_{t \in [0,T], r \in \R} e^{-a |r|} \E \left [e^{-2h^k_t(r)} \right ] < \infty \mbox { for some } a \in (0, \infty).
$$
Furthermore, the paper \cite{BertiniGiacomin} shows that as $k \to \infty$, $h^k$ converges in distribution (as a $C([0,T]: C(\R))$ valued random variable) to a limit process $h$, which is {\em defined} to be the
solution of \eqref{eq:kpz}.  Throughout this work, this process (strictly speaking -- its probability law on $C([0,T]: C(\R))$) will be referred to as
{\em the BGC solution of the KPZ equation}.

We will now introduce a forward - backward stochastic differential equation associated with \eqref{eqn: Condition_Initial}.
Assume, without loss of generality, that we are given on $(\Om, \clf, \Pp)$ another real standard Brownian motion $W$ that is independent of $(B, h_0)$.  For $S > 0$ and $0 \le t \le S$, we denote
$$
\clf_{t,S}^W = \sigma \{W(s) -W(t): s \in [t,S]\}\vee \cln, \mbox{
and } \clf_t^S = \clf_{t,S}^W \vee \clf_{t}^B \vee \sigma\{h_0\}.$$
Note that $\clf^S \equiv \{\clf_t^S: t \in [0, S]\}$ is not a
filtration since the $\sigma$-fields in this collection are neither
increasing nor decreasing in $t$.  However, abusing terminology, we
will say a stochastic process $\{V_t\}_{t\in [0, S]}$ is $\clf^S$
adapted if $V_t$ is $\clf_t^S$ measurable for every $t \in [0, S]$.

Throughout we will fix a complete orthonormal system
$\{\gamma_m\}_{m \in \mathbb{N}}$ in $\Hh$ and denote $B(\gamma_m) =
\beta_m$.  Note that $\{\beta_m\}_{m \in \mathbb{N}}$ is a sequence
of independent standard Brownian motions, independent of $W$. For a
$\clf^S$ adapted $\Hh$-valued process $\{\varphi(t)\}_{0 \le t \le
S}$ satisfying $\E \int_0^S \|\varphi(t)\|^2 dt < \infty$, the
It\^{o} stochastic integral $\int_0^t \langle \varphi(r),
dB_r\rangle$ for $t \in [0, S]$ is well defined and is given as
\[
\int_0^t \langle \varphi(r), d B_r \rangle = \sum_{ m \in \N } \int_0^t \langle \varphi(r), \gamma_m \rangle d\beta_m(r),
\]
where the series on the right converges in $L^2(\Pp)$.

For a family of sigma algebras $\xbr{\clg_t; 0 \le t \le S }$, let $\mathscr{H}^p_S (  \clg ) $ [ resp. $\mathscr{H}^\infty_S ( \clg )$] be the space of measurable [resp. continuous] processes $\{ \phi(t): t \in [0,S] \}$ such that
$\phi(t)$ is $\clg_t$ measurable for every $t$, and
\[
\E \int_0^S | \phi(t) |^p dt < \infty \; [\mbox{ resp. } \E \sup_{ t \in [0,S]} | \phi(t) |^2 < \infty.]
\]

For $H \in \mathscr{H}^2_S (  \Fc^S )$, we denote the backward
stochastic integral of $H$ with respect to $W$ by $\int_t^S H(r)
\downarrow dW$. See Appendix for a brief introduction to such
stochastic integrals.

Let $X^S_t(x) \equiv X^S(t,x) = x + W(S) - W(t)$, for $0 \le t \le S$ and $x \in \R$.
Define
\be \label{eq:martterm}
Z^k(t,x) = \int_0^t \langle \zz^k_{X^S(r,x)}, dB_r \rangle, \; (t,x) \in [0,S] \times \R.\ee
Note that $\|\zz^k_x\|^2 = \|\zz^k_0\|^2 = C^k(0)$ for all $x$ and consequently 
\be
\label{ab824}
\int_0^t \|\zz^k_{X^S(r,x)}\|^2 dr = tC^k(0)  \mbox{ for all } t \in [0, S]. \ee
Also, $\{\zz^k_{X^S(t,x)}\}_{0\le t \le S}$ is $\clf^S$ adapted and so the stochastic integral
in \eqref{eq:martterm} is well defined and has the representation
$$
Z^k(t,x) = \sum_{m \in \N} \int_0^t \langle \zz^k_{X^S(r,x)}, \gamma_m \rangle d\beta_m(r)$$
with the series converging in $L^2(\Pp)$.
We now consider the following doubly backward SDE
\begin{align}  \notag
y_S^k (t,x) &= h_0( X_0^S (x) ) - \frac{1}{2} \int_0^t \parbar{ z_S^k(r,x)^2 - C^k(0) } dr + Z^k(t,x) \\
& \quad   - \int_0^t z_S^k(r,x) \downarrow dW(r). \label{eqn: DBSDE_d}
\end{align}
In such an equation one needs to solve for a pair of real stochastic processes $\{y_S^k(t,x), z^k_S(t,x)\}_{0\le t \le S}$ with suitable
integrability and measurability properties such that the terms on the right side of the equation are meaningful and the equation holds a.s. for all $t \in [0, S]$.
Frequently, when clear from the context, we will suppress $x$ and denote the solution as $\{y^k_S(t), z^k_S(t)\}_{0\le t \le S}$ or merely as $(y^k_S, z^k_S)$.
The first result in this work establishes wellposedness of the above equation.
\begin{theorem} \label{thm: ExistenceDBSDEMain}
Fix $x \in \R, k \in \N$ and $S>0$. Suppose that $h_0$  is a $C(\R)$ valued random variable, independent of $(B,W)$, satisfying \eqref{eqn: Condition_Initial}.
 Then there is a unique pair $(y_{S}^k,z_S^k) \in \mathscr{H}_S^\infty (\Fc^S) \times \mathscr{H}_S^2 ( \Fc^S )$  that satisfies equation~\eqref{eqn: DBSDE_d}.
\end{theorem}
\begin{remark}
    By uniqueness of $(y_{S}^k,z_S^k)$, we mean that if $(y,z), (\tilde y, \tilde z) \in \mathscr{H}_S^\infty (\Fc^S) \times \mathscr{H}_S^2 ( \Fc^S )$ are two
    solutions of \eqref{eqn: DBSDE_d} then
    $(y(t), z(t)) = (\tilde y(t), \tilde z(t))$, a.e. $t \in [0,S]$, a.s.
\end{remark}

Our second result concerns the asymptotic behavior of $y^k_S$, as $k \to \infty$, and relation with the KPZ equation.
\begin{theorem} \label{thm: asympink}
Fix $x \in \R$ and $ k \in \N$.  Let $h_0$ be as in Theorem \ref{thm: ExistenceDBSDEMain} and, for $k \ge 1$
and $S>0$, $(y_S^k,z_S^k)$ be as obtained from Theorem \ref{thm: ExistenceDBSDEMain}. Then, the sequence $y^k_S$ is tight in the space $C([0,S]: C(\R))$. 

Furthermore, for any $T> 0$ and $S \in [0,T]$, if we let  $\hat h^k(S,x) = y^k_S(S,x)$, then $\hat h^k$ is a $C([0,T]: C(\R))$ valued random variable such that $\hat h^k$ converges in distribution, as $k \to \infty$, to the BGC solution of the KPZ equation.
\end{theorem}

The rest of this work is devoted to the proof of the above two results.

\section{Proof of Theorem \ref{thm: ExistenceDBSDEMain}.}
Fix  $k \in \N$ and $S>0$.  Suppose that $(y^k_S, z^k_S)$ solves equation \eqref{eqn: DBSDE_d}.
Define
\be
\label{eq:uv}
u^k_S(t,x) = \exp(-y^k_S(t,x)), \;\; v^k_S(t,x) = u^k_S(t,x)z^k_S(t,x), \; 0 \le t \le S, \; x \in \R .
\ee
A formal application of It\^{o}'s formula (Lemma~\ref{lemma: Ito})  yields the following equation for $(u^k_S, v^k_S)$.
\begin{align}
u^k_S(t,x) &=  u_0( X^S (0,x) ) -  \int_0^t  u^k_S(r,x) d Z^k(r,x) + \int_0^t v^k_S(r,x) \downarrow dW(r),
\label{eqn: HopfCole}
\end{align}
where $u_0(x) = \exp(-h_0(x))$.
The transformation in \eqref{eq:uv} thus motivates the study of equation \eqref{eqn: HopfCole}, and as a first step we will now establish
 the wellposedness of \eqref{eqn: HopfCole}. Namely, we first prove the existence and uniqueness of a pair $(u^k_S, v^k_S)$, with appropriate integrability and measurability properties, which satisfies
\eqref{eqn: HopfCole}.  Note that the integrals on the right side of \eqref{eqn: HopfCole}
are well defined if $(u_S^k,v_S^k) \in \mathscr{H}_S^\infty (\Fc^S) \times \mathscr{H}_S^2 ( \Fc^S )$.
\begin{lemma} \label{lemma: ExistenceHopfCole} Fix $x \in \R, k \in \N$ and $S>0$. Then there is a unique pair
    $(u_S^k,v_S^k) \in \mathscr{H}_S^\infty (\Fc^S) \times \mathscr{H}_S^2 ( \Fc^S )$  that satisfies equation~\eqref{eqn: HopfCole}. Furthermore,
\begin{align}
u_S^k(t,x) &= \E \xcmd{ u_0( X^S (0,x) ) \exp \left\{ - Z^k (t,x) -\frac{1}{2}C^k(0)t \right\} }{\Fc^S_t}  \label{eqn: FeynmanKac}.
\end{align}

Finally, for any $p\geq 2$, there is a  $C(p,k) \in (0, \infty)$ such that, for all $x\in \R$,
\begin{equation} \label{eqn: H_p}
\E \sup_{ t \leq S } u_S^k(t,x)^p + \E \left( \int_0^t v_S^k(r,x)^2 dr \right)^{ p/2} \leq C(p,k) ( 1+  \E u_0 ( X^S (0,x) )^{ 4 p} ).
\end{equation}
\end{lemma}
{\bf Proof.}
  Since $x, k$ and $S$ are fixed, we omit them from the notation throughout this proof. In particular we write
$Z^k(t,x)$ and $X^S(t,x)$ as $Z(t)$ and $X(t)$  respectively.
For a stochastic  process $H = \{H(t)\}_{0 \le t \le S}$, we define its time reversed path
 $$\tilde H(t)=H(S-t)-H(S), \; 0 \le t \le S.$$
 In particular,
$$\tilde B_t(f) = B_{S-t} (f) - B_S (f),\; f \in \Hh, \mbox{ and } \tilde W(t) =  W(S-t)- W(S).$$
Define $\hat X(r)= x - \tilde{W}(r)$.  Then note that
\begin{align}
    \tilde Z(t) =& Z(S-t) - Z(S)
    = -\int_{S-t}^{S} \langle \zz^k_{X(r)}, dB_r\rangle \nonumber \\
    =& -\sum_m\int_{S-t}^{S} \langle \zz^k_{X(r)}, \gamma_m\rangle d\beta_m(r)
    =\sum_m\int_{0}^{t} \langle \zz^k_{\hat X(r)}, \gamma_m\rangle d\tilde\beta_m(r)\nonumber \\
    =& \int_{0}^{t} \langle \zz^k_{\hat X(r)}, d\tilde B_r\rangle, \label{eq:timrev}
\end{align}
where the next to last equality follows on noting that
$$ \hat X(S-r) = x - \tilde W(S-r) = x - W(r) + W(S) = X(r).$$
Let, for $0 \le t \le s \le S$,
$$
\clf_{t,s}^{\tilde B} = \sigma \{\tilde B_s(v) - \tilde B_u(v): u
\in [t,s], v \in \Hh\}\vee \cln 
$$ and
\begin{equation} \label{eq:1233a}\clf_t^{\tilde W} = \sigma\{\tilde W(s): 0 \le s \le t\}\vee \cln = \clf_{S-t,S}^{W}.\end{equation}
Note in particular that
\begin{align}\clf_{t,S}^{\tilde B} =& \sigma \{\tilde B_S(v) - \tilde B_u(v): u
\in [t,S], v \in \Hh\}\vee \cln \nonumber\\
= & \sigma \{  B_u(v): u \in [0,S-t], v \in \Hh\}\vee \cln =
\clf_{S-t}^B \label{eq:1233}
\end{align}
Also, let
\begin{equation} \label{eqn: TildeFS}
\tilde \clf_t^S = \clf_t^{\tilde W} \vee \clf_{t,S}^{\tilde B}\vee
\sigma \{h_0\} = \clf_{S-t,S}^W
 \vee \clf_{S-t}^B \vee \sigma \{h_0\}= \clf_{S-t}^S.
 \end{equation}
From Corollary \ref{lemma: FBRelation} in the Appendix it follows
that in order to prove the first statement of the lemma it suffices
to show that there exists a unique pair $(\hat u, \hat v) \in
\mathscr{H}_S^\infty ( \tilde \clf^S)   \times \mathscr{H}_S^2 (
\tilde \clf^S )$ that solves the time reversed equation
\begin{equation} \label{eqn: TimeReversalLinear}
\hat u (t) = u_0 ( \hat X(S) ) + \int_t^S \hat u(r) \downarrow d \tilde Z (r ) - \int_t^S \hat v(r) d\tilde W(r).
\end{equation}
The unique solution $(u,v)$ of \eqref{eqn: HopfCole} can then be obtained on taking $(u(t), v(t)) = (\hat u(S-t), \hat v(S-t))$.
We now consider the unique solvability of \eqref{eqn: TimeReversalLinear}.
Let  $\clg_t = \clf_t^{\tilde W} \vee \clf_{0,S}^{\tilde B}\vee \sigma\{h_0\}$.
Recalling that $\|\zz^k_x\|^2 = C_k(0)$ for all $x \in \R$ and elementary properties of Brownian motions, we see that
\be \label{eq:ab1905}  \E\sup_{0 \le t \le S}\exp \{m |Z(t)|\}  < \infty , \; \E \exp \{m |\hat X(S)|\} < \infty , \; \mbox{ for all } m \in \N .\ee
Combining this with the integrability condition \eqref{eqn: Condition_Initial} and an application of Cauchy-Schwarz inequality we have
$$ \E \left [u_0 (  \hat X(S)  ) \exp\{  \tilde Z(S) \} \right]^2 < \infty .$$
Consequently,
\be
\label{eq:basmart}
M(t) = \E \left [ u_0(\hat X(S))\exp \{\tilde Z(S) - \frac{1}{2}C^k(0)S \} \mid \clg_t \right], \; 0 \le t \le S,\ee
is a square integrable $\clg_t$ - martingale.
From a straightforward extension of the classical martingale representation theorem, there is a  $\clg_t$-progressively measurable process $\xbr{ J(t); 0 \leq t \leq S }$ such that
$\E \int_0^S J(r)^2 dt < \infty$, and
\begin{equation} \label{eqn: Jdef}
M(t) = M(0) + \int_0^t J(r) d\tilde W(r), \; 0 \le t \le S.
\end{equation}
Define, for $0 \le t \le S$, 
 \be \label{eq:uv2} E(t) = \exp\{ -\tilde Z(t) + \frac{1}{2}C^k(0)t\},\; V(t)=  E(t)J(t),
\; U(t) = E(t)M(t).\ee
We now show that $(U,V) \in \Hspace{\infty}{\tilde \clf^S} \times \Hspace{2}{\tilde \clf^S}$.
More precisely, for the process $V$ we will show that there is a $\tilde V \in \Hspace{2}{\tilde \clf^S}$ such that
$\tilde V(t) = V(t)$ for a.e. $t$, a.s.
Consider $U$ first.  Note that $U$ can be rewritten as
\begin{align}
    U(t) =& E(t) \left (M(S) - \int_t^S J(r) d\tilde W(r)\right ) \label{eq:537new}\\
    =& u_0(\hat X(S))\frac{E(t)}{E(S)} - E(t)\int_t^S J(r) d\tilde W(r).\nonumber
\end{align}
From \eqref{eq:uv2} and \eqref{eq:timrev} it follows that $U$ is $\{\clg_t\}$ adapted.  Next, using the independence between
$\tilde Z(S)-\tilde Z(t)$ and $\clf_{0,t}^{\tilde B}$ and that $\{M(t)\}$ is a $\{\clg_t\}$ martingale, we see from the above display that
\begin{align}
U(t) =& \E(U(t)\mid \clg_t) \nonumber\\
 =& \E \left [u_0(\hat X(S))\frac{E(t)}{E(S)} \mid \clg_t \right ] \label{eq:ab1903}\\
=& \E \left [u_0(\hat X(S))\frac{E(t)}{E(S)} \mid \clf_t^{\tilde W} \vee \clf_{t,S}^{\tilde B}\vee \sigma\{h_0\} \right ].\label{eq:fk2}
\end{align}
Thus $U$ is $\tilde \clf^S$ - adapted.  We now argue that there is a $\tilde V$ that is $\tilde \clf^S$ - adapted and such that
$V(t) = \tilde V(t)$, for a.e. $t$, a.s.
For $c \in (0, \infty)$, define $J^c(r) = J(r) 1_{|J(r| \le c)}$.
Let, for $\varepsilon > 0$
$$F^c_{\varepsilon} = \frac{1}{\sqrt{\varepsilon}} \int_{t}^{t+\varepsilon} J^c(r) d\tilde W(r).$$
By It\^o's isometry, we have that
$$
\E\left [ F_\varepsilon^c \frac{\tilde W(t+\varepsilon) - \tilde W(t)}{\sqrt{\varepsilon}} \mid \clg_t\right] = \E \left [\frac{1}{\varepsilon}\int_{t}^{t+\varepsilon} J^c(r) dr \mid \clg_t\right].$$
Sending $\varepsilon \to 0$, $c \to \infty$, and recalling that $J(t)$ is $\clg_t$ measurable, we have that
$$
\limsup_{c \to \infty}\limsup_{\varepsilon \to 0}\E\left [ F^c_\varepsilon \frac{\tilde W(t+\varepsilon) - \tilde W(t)}{\sqrt{\varepsilon}} \mid \clg_t\right] = J(t) = \frac{V(t)}{E(t)},
\; \mbox{ a.e. } t, \; \mbox{ a.s. }$$
and therefore since $E(t)$ is $\clg_t$ measurable
\be \label{ab1820}
V(t) = \limsup_{c \to \infty}\limsup_{\varepsilon \to 0}\E\left [ E(t)F^c_\varepsilon \frac{\tilde W(t+\varepsilon) - \tilde W(t)}{\sqrt{\varepsilon}} \mid \clg_t\right]\; \mbox{ a.e. } t, \; \mbox{ a.s. }.\ee
Define $U^c$ by replacing $J$ with $J^c$ on the right side of \eqref{eq:537new}.  Then a calculation similar to the one leading to \eqref{eq:fk2} shows that $U^c$ is $\tilde \clf^S$ - adapted.

Also note that
$$
\sqrt{\varepsilon} E(t)F^c_{\varepsilon} = \exp\{\tilde Z(t+\varepsilon) - \tilde Z(t) -\frac{1}{2}C^k(0) \varepsilon\}U^c(t+\varepsilon) - U^c(t)$$
and consequently $\sqrt{\varepsilon} E(t)F^c_{\varepsilon}$ is independent of $\clf_{0,t}^{\tilde B}$.  Thus the right side of \eqref{ab1820} equals
$$
\limsup_{c \to \infty} \limsup_{\varepsilon \to 0}\E\left [ E(t)F^c_\varepsilon \frac{\tilde W(t+\varepsilon) - \tilde W(t)}{\sqrt{\varepsilon}} \mid \tilde \clf_t^S\right]$$
and so $\tilde V$ defined by the right side of \eqref{ab1820} is $\tilde \clf^S$ - adapted.

We now prove the stated integrability properties of $(U,V)$.  From \eqref{eq:ab1903}, for $m \in \N$,
\be
\label{eq:ab1911}
\E \sup_{0\le t \le S} U(t)^m \le \E \left [\sup_{0\le t \le S}  \E \left[\left(\frac{u_0(\hat X(S))^m}{E(S)^m}\sup_{0\le r \le S}E(r)^m\right)  \mid \clg_t \right]\right].
\ee
From  \eqref{eq:ab1905} it follows that, for any $m\ge 1$,
$$ \E\left[\frac{u_0(\hat X(S))^m}{E(S)^m}\sup_{0\le r \le S}E(r)^m\right] < \infty.$$
Standard martingale inequalities now show that the right side of \eqref{eq:ab1911} is finite, indeed we have  that, for every $m \in \N$
there are  $C_1(m), C_2(m) \in (0, \infty)$ such that
\begin{align}
\E \sup_{0\le t \le S} U(t)^m \le & C_1(m)\left(1+ \E\left[\frac{u_0(\hat X(S))^{2m}}{E(S)^{2m}}\sup_{0\le r \le S}E(r)^{2m}\right]\right)\nonumber\\
\le & C_2(m) \left(1+ \E u_0((\hat X(S)) ^{4m}\right) \label{eq:ab1917} \\
<& \infty . \nonumber
\end{align}
Next consider $V$.  By classical martingale inequalities (cf. Proposition 3.3.26 of \cite{KaratzasBook}), for every $m \in \N$ there is a
$b_m \in (0,\infty)$ such that
$$
\E \left ( \int_0^S J(r)^2 dr\right)^m \le b_m \left( \E M(S)^{2m} + \E M(0)^{2m}\right).
$$
Thus, recalling the definition of $\{M(t)\}$  (see \eqref{eq:basmart} ) and using \eqref{eq:ab1905} once more, we have that for every
$m \in \N$ there is a $C_3(m) \in (0, \infty)$, such that
\begin{align}
\E \left ( \int_0^S J(r)^2 dr\right)^m \le & C_3(m) \left(1 + \E (u_0(\hat X(S))^{4m} )\right) \label{eq:bdsi} \\
< & \infty  . \nonumber
\end{align}
Next,
\begin{align*}
    \E \left ( \int_0^S V(r)^2 dr\right)^m =& \E \left ( \int_0^S E(r)^2 J(r)^2 dr\right)^m\\
    \le & \E \left[\left ( \sup_{0\le r \le S}E(r)^{2m}\right) \left(\int_0^S  J(r)^2 dr\right)^m\right].
\end{align*}
Finiteness of the last term is immediate from \eqref{eq:ab1905} and \eqref{eq:bdsi}.  In fact we have that
for every
$m \in \N$ there is a $C_4(m) \in (0, \infty)$, such that
\be
\label{eq:bdv}
\E \left ( \int_0^S V(r)^2 dr\right)^m  \le C_4(m) \left(1 + \E (u_0(\hat X(S)) ^{8m} )  \right).
\ee
Combining \eqref{eq:ab1917}, \eqref{eq:bdv} and the $\tilde \clf^S$ adaptedness of $(U,V)$ we have in particular that
$(U,V) \in  \Hspace{\infty}{\tilde \clf^S} \times \Hspace{2}{\tilde \clf^S}$.
%
By an application of It\^{o}'s formula, we now see that, for $t \in [0, S]$
\be \label{ab850}
U(t)=U(S) + \int_t^S U(r) \downarrow d\tilde Z(r) - \int_t^S V(r) d\tilde W(r).
\ee
For completeness, we give a proof  of the above equality in the Appendix.

Thus we have shown that $(U,V)$ is a solution of \eqref{eqn: TimeReversalLinear} and therefore, as noted earlier
\be
\label{ab757}(u^k_S(t), v^k_S(t)) \equiv (U(S-t), V(S-t)) \ee
defines a solution of \eqref{eqn: HopfCole}.
Representation \eqref{eqn: FeynmanKac} for the solution $u^k_S$ is immediate from the definition of $\tilde Z$ and \eqref{eq:fk2}.
Also, it follows from \eqref{eq:ab1917} and \eqref{eq:bdv} that the solution satisfies \eqref{eqn: H_p}
for any $p \ge 2$.

We now prove uniqueness. Let $(u,v), (u',v') \in \mathscr{H}_S^\infty ( \tilde \clf^S)   \times \mathscr{H}_S^2 ( \tilde \clf^S)$ be two solutions of~\eqref{eqn: TimeReversalLinear}. Then, the differences $\xi=u-u^\prime$, and $\eta=v-v^\prime$ satisfy 
\[
\xi(t)= \int_t^T  \xi(r) \downarrow d\tilde Z(r ) - \int_t^T \eta(r) d\tilde W (r).
\]
Using Lemma \ref{lemma: Ito} (ii) in the Appendix, we get that
\[
\xi(t)^2 + \int_t^S \eta(r)^2 dr= 2 \int_t^S  \xi(r)^2 \downarrow d\tilde Z (r ) - 2\int_t^S \eta(r) \xi(r) d\tilde W (r) + C^k (0 ) \int_t^S \xi(r)^2 dr.
\]
Taking expectations and using Gronwall's inequality it follows that
\[
\E \xi(t)^2 + \E \int_t^S \eta(r)^2 dr =0,
\]
The unique solvability of \eqref{eqn: TimeReversalLinear}, and consequently that of \eqref{eqn: HopfCole} follows.
This completes the proof of the lemma. $\Box$

\begin{proof}[Proof of Theorem~\ref{thm: ExistenceDBSDEMain}]
As in the proof of Lemma \ref{lemma: ExistenceHopfCole}, we will suppress $n,k,S$ from the notation, unless necessary.
Let $(u,v) \in \Hspace{\infty}{\Fc^S} \times \Hspace{2}{\Fc^S}$ be the solution of~\eqref{eqn: HopfCole}.
We will obtain a solution of \eqref{eqn: DBSDE_d} by taking the logarithmic transform of $u$.  We begin by showing that
\be
\label{posalw}
\inf_{0\le t \le S} u(t) > 0, \; a.s.
\ee
Recall from \eqref{eq:uv2} that $u(t) = E(S-t)M(S-t)$, $0 \le t \le S$.  Clearly $\inf_{0\le t \le S} E(S-t) > 0$.  
Also, from the expression of $M(t)$ given in \eqref{eq:basmart} we see that, for each $t$, $M(t) > 0$ a.s., since the random variable inside the conditional expectation
is strictly positive a.s.  Also, since $M$ is continuous, we have that 
%
$\inf_{0\le t \le S} M(t) > 0$ a.s. Combining these observations we see that \eqref{posalw} holds.
 Define
\begin{equation} \label{eqn: InverseHopfCole}
y(t) = - \log u(t),\text{ and } z(t) =  \frac{ v (t) }{ u (t) }.
\end{equation}
We now argue that
$(y,z) \in \Hspace{\infty}{\Fc^S} \times \Hspace{2}{\Fc^S}$. For  $y$ note that
%
\begin{align} \notag
\E \sup_{ t \in [0,S] } y(t)^2 &\leq \E \sup_{ t \in [0,S] } y(t)^2 \ONE_{ u(t) \leq 1 } + \E \sup_{ t \in [0,S] } y(t)^2 \ONE_{ u(t) > 1 } \\
&\equiv T_1 + T_2, \label{eqn: T1T2}
\end{align}
Using the inequality $0<\log \theta < \theta$ for all $\theta >1$,
\begin{align*}
T_2 & \leq \E \sup_{ t \in [0,S] } u(t)^2< \infty.
\end{align*}
Next consider $T_1$.
From \eqref{eq:ab1903}, \eqref{ab757} and an application of Jensen's inequality we have that
\begin{align*}
|y(S-t)| \ONE_{ u(S-t) \leq 1 } &= - \log \parbar{ U(t)\ONE_{ U(t) \leq 1 } + \ONE_{ U(t) > 1 } }  \\
& = -  \log \E\xcmd{ u_0( \hat X(S) ) \frac{E(t)}{E(S)} \ONE_{ U(t) \leq 1} + \ONE_{ U(t) \ge 1}} {\clg_t}\\
& \leq - \E \xcmd{ \log \parbar{ u_0( \hat X(S) ) \frac{E(t)}{E(S)} \ONE_{ U(t) \leq 1}  + \ONE_{ U(t) > 1 } }} {\clg_t}\\
&= - \E \xcmd {\log \parbar{ u_0( \hat X(S) ) \frac{E(t)}{E(S)}    }} {\clg_t}\ONE_{ U(t) \leq 1}.
\end{align*}
Recalling that $u_0 = \exp\{-h_0\}$, we have
$$
|y(S-t)| \ONE_{ u(S-t) \leq 1 } \le \E \xcmd{|h_0(\hat X(S))|} {\clg_t} + \E \xcmd{|\tilde Z(t) - \tilde Z(S)|} {\clg_t} + \frac{1}{2} C^k(0) (S-t).$$
Recalling that $\{\clg_t\}$ is a filtration and that from \eqref{eq:ab1905} and \eqref{eqn: Condition_Initial}
$$ \E \left( \sup_{0\le t \le S} |\tilde Z(t)|^2 + |h_0(\hat X(S))|^2\right ) < \infty, $$
we have by an application of Doob's inequality that for some $C_1 \in (0, \infty)$
$$
T_1 = \E \sup_{ t \in [0,S] } y(t)^2 \ONE_{ u(t) \leq 1 } = \E \sup_{ t \in [0,S] } y(S-t)^2 \ONE_{ u(S-t) \leq 1 } < \infty .$$
Using the above estimates on $T_1$ and $T_2$ in \eqref{eqn: T1T2} we see that $y \in \Hspace{\infty}{\Fc^S}$.

We now consider $z$. Let $p\geq 2$ and $q$ be such that $p^{-1} +
q^{-1}=1$. Then using Holder's inequality
\begin{align} \notag
\E \int_0^S z(r)^2 dr &= \E \int_0^S \parbar{ \frac{v(r)}{u(r)}}^2 dr \\ \notag
& \le  \E  \sup_{ t \in [0,S] } u(t)^{ -2 } \int_0^S v(r)^2dr \\
& \leq \parbar{\E  \sup_{ t \in [0,S] } u(t)^{ -2p }}^{p^{-1} } \parbar{ \E \parbar{ \int_0^S v(r)^2dr }^q }^{q^{-1}}.  \label{eqn: VInt_1}
\end{align}
From \eqref{eq:ab1903} and Jensen's inequality
\begin{align} \notag
U(t)^{-2p} &= \left (\E \xcmd{U(t)}{\clg_t}\right)^{-2p} \\ \notag
& =  \left (\E \xcmd{u_0(\hat X(S)) \frac{E(t)}{E(S)}}{\clg_t}\right)^{-2p} \\
& \leq \E \xcmd{\left(u_0(\hat X(S))\right)^{-2p} \frac{E(S)^{2p}}{E(t)^{2p}}}{\clg_t}.   \label{eqn: ab907}
\end{align}
Recalling \eqref{ab757}, we have that
\begin{align} \notag
\E \sup_{0\le t \le S} u(t)^{-2p} &= \E \sup_{0\le t \le S} U(t)^{-2p} \\ 
& \leq \E \sup_{0\le t \le S} \E\xcmd{\left(u_0(\hat X(S))\right)^{-2p}E(S)^{2p}\sup_{0\le r \le S} E(r)^{-2p}}{\clg_t}.   \label{eqn: ab911}
\end{align}
Also, from \eqref{eq:ab1905} and \eqref{eqn: Condition_Initial}
$$
\E\left [\left(u_0(\hat X(S))\right)^{-4p}E(S)^{4p}\sup_{0\le r \le
S} E(r)^{-4p}\right] < \infty .$$ Since $\{\clg_t\}$ is a
filtration, we have that the conditional expectation
in \eqref{eqn: ab911} is a martingale and so by Doob's maximal inequality it
follows that $$\E \sup_{0\le t \le S} u(t)^{-2p} < \infty.$$ Combining
this estimate with \eqref{eq:bdv}, \eqref{eqn: VInt_1} and recalling
\eqref{ab757}, we have that $z \in  \Hspace{2}{\Fc^S}$.

To finish the proof of existence of solutions, we now verify that
$(y,z)$ defined in \eqref{eqn: InverseHopfCole} satisfy~\eqref{eqn: DBSDE_d}.
We will apply Lemma~\ref{lemma: Ito} (i)  with $\alpha = u, \beta =0, \gamma = -u, \delta = v$, and $\phi(x) = -\log(x)$.  Note that although $\phi$ is only $C^2$ on $(0, \infty)$,  ~\eqref{posalw} guarantees the applicability of It\^{o}'s formula.
Representation~\eqref{eqn: HopfCole}, and Lemma~\ref{lemma: Ito}  imply that
\begin{align*}
y(t) &=y(0) + \int_0^t \frac{u(r)}{u(r)} dZ(r) - \int_0^t \frac{v(r)}{u(r)} \downarrow dW(r)+ \int_0^t \frac{u(r)^2}{2 u(r)^2} C^k (0) dr - \int_0^t  \frac{v(r)^2}{2 u(r)^2}  dr   \\
&= y(0) + Z(t) -  \int_0^t \frac{v(r)}{u(r)} \downarrow dW(r) - \frac{1}{2} \int_0^t \parbar{ \parbar{ \frac{v(r)}{u(r)} }^2 - C^k(0)  }dr.
\end{align*}
From the fact that $z(t)=v(t)/u(t)$, and $y(0)=h(X^S(0,x) )$ we see that this equation is the same as~\eqref{eqn: DBSDE_d}. This completes the proof of existence. 

We now prove uniqueness. Suppose $(y,z), (\tilde y, \tilde z) \in \mathscr{H}_S^\infty (\Fc^S) \times \mathscr{H}_S^2 ( \Fc^S )$ are two solutions
of \eqref{eqn: DBSDE_d}.  Let $(\bar y, \bar z) = (y-\tilde y, z - \tilde z)$.  For $M \in (0, \infty)$, define
$\psi_M: \R \times \R \to [-M, M]$ as $\psi_M(a,b) = \frac{1}{2}(2a-b)\ONE_{|2a-b| \le M}$.  Let $\varphi_M(r) = \psi_M(z(r), \bar z(r))$, $r\in [0,S]$ and let
$y_M$ be a continuous process defined as
\be \label{ab:912}y_M(t) = - \int_0^t \bar z(r) \varphi_M( r ) dr - \int_0^t \bar z(r) \downarrow dW(r).\ee
We will now show that
\be
\label{eq:ab800}
y_M(t) = 0, \mbox { a.s.  for all } t \in [0, S] \mbox{ and } M \in (0, \infty).
\ee
Note that if \eqref{eq:ab800} holds, we have on sending $M \to \infty$, and observing that $y_M(t) \to \bar y(t)$ in probability, for every $t \in [0, S]$, that
$y$ and $\tilde y$ are indistinguishable. Moreover, an application of It\^{o}'s formula (see Lemma \ref{lemma: Ito}(i)) shows that
$$
y_M^2(t) = - 2\int_0^t y_M(r)\bar z(r) \varphi_M( r ) dr - 2\int_0^t y_M(r)\bar z(r) \downarrow dW(r) - \int_0^t \bar z^2(r) dr$$
and so if \eqref{eq:ab800} holds, we have that $z(t) = \tilde z(t)$, a.e. $t \in [0, S]$, a.s.  Combining the above observations we see that in order to prove uniqueness, it suffices to verify \eqref{eq:ab800}.

From Tanaka's formula (cf. Theorem IV.68 in \cite{Protter}) it follows that
\begin{align*}
\parbar{ y_M(t) }_+ &= - \int_0^t \ONE_{ \{y_M(r)>0\} } \bar z(r) \varphi_M( r ) dr \\
&- \int_0^t \ONE_{ \{y_M(r)>0\} } \bar z(r) \downarrow dW(r) - \frac{1}{2}L^0(t),
\end{align*}
where $y_+ = \max \xbr{y,0}$, and $L^0$ is the local time at $0$ process for $Y_M$ (see Chapter IV of \cite{Protter}).
In particular, $L^0$ is non-decreasing, non-negative process such that
\be
\label{eq:ab815}
\int_{[0,\infty)} \ONE_{\{y_M(t) > 0\}} dL^0(t) = 0.
\ee
We remark that the cited theorem establishes the above formula for equations with forward stochastic integrals, however the version with backward integrals used
here follows by straightforward modifications of the proof.

Define for  $n \in \N$, $\xi_n: [0, \infty) \to [0, \infty)$ as
\[
\xi_{n} (u) = ( u \wedge n )^2 + 2n ( u - n )_+, \quad u \in [0,\infty).
\]
Then $\xi_n$ is a $C^1$- convex function with
\be
\label{eq: ab819}
\xi'_n(u) = 2(u\wedge n), \; u \in [0, \infty).
\ee
By Meyer-It\^{o} formula (cf. Theorem IV.70 in \cite{Protter})
\begin{align} \label{eq:ab843}
\xi_{n} ( (y_M (t))_+ )
&= - \int_0^{t } \xi'_n((y_M(s))_+)\ONE_{\{y_M(s) > 0\}} \bar z(r) \varphi_M(r) dr  \\ \notag
&- \int_0^{t } \xi'_n((y_M(s))_+)\ONE_{\{y_M(s) > 0\}} \bar z(r) \downarrow dW(r) \\ \notag
& - \frac{1}{2} \int_0^{t }\xi'_n((y_M(s))_+) dL^0(s) - \frac{1}{2} \int_{-\infty}^{\infty} L^a(t) \mu(da).
\end{align}
where $L^a$ is the local time process of $(y_M)_+$ at level $a$ and $\mu$ represents the second derivative of $\xi_n$ in the generalized function sense.  Since $\xi_n$ is convex, $\mu$ is a (non-negative) measure and in fact equals
$$\mu(da) = 2\ONE_{[0,n]}(a) da.$$
Thus (cf. Corollary IV.1 of \cite{Protter})
$$
\int_{-\infty}^{\infty} L^a(t) \mu(da) = 2 \int_0^n L^a(t) da = 2 \int_0^t \ONE_{(0,n]}((y_M(s))_+) \bar z^2(s) ds. $$
Combining this with the fact that the third term on the right side of \eqref{eq:ab843} is zero, we have from \eqref{eq: ab819} that
\begin{align} \label{eq:ab905}
\xi_{n} ( y_M (t)_+ )
&+  \int_0^t  \ONE_{(0,n]}((y_M(r))_+)  \bar z^2(r) dr \\ \notag
&= - 2\int_0^{t }\ONE_{(0,n]}((y_M(r))_+) ( y_M (r) )_+ \bar z(r) \varphi_M( r ) dr  \\ \notag
& \quad-2n \int_0^t\ONE_{(n,\infty)}((y_M(r))_+) \bar z(r) \varphi_M( r ) dr \\ \notag
& \quad -  \int_0^t\ONE_{ \{y_M(r)>0\} }\xi_{n}^\prime ( (y_M (r))_+ ) \bar z(r) \downarrow dW(r).
\end{align}
Using Young's inequality we have that, for any $\alpha > 0$,
\begin{align} \notag
\int_0^{t }\ONE_{(0,n]}((y_M(r))_+) ( y_M (r) )_+ |\bar z(r)|\; |\varphi_M( r )| dr
& \le  \frac{\alpha}{2} \int_0^{t } \ONE_{(0,n]}((y_M(r))_+) |\bar z(r)|^2 dr \\ \notag
&+  \frac{1}{2\alpha} \int_0^{t }( y_M (r) )_+^2 |\varphi_M( r )|^2 dr.
\end{align}
Using the above estimate with $\alpha < 1$ in \eqref{eq:ab905}, we have
\begin{align} \label{ab922a}
\xi_{n} ( y_M (t)_+ ) &\le \frac{M^2}{\alpha} \int_0^{t }( y_M (r) )_+^2 dr \\ \notag
&+ 2nM \int_0^t\ONE_{(n,\infty)}((y_M(r))_+) |\bar z(r)|  dr\\ \notag
&-\int_0^t\ONE_{ \{y_M(r)>0 \}}\xi_{n}^\prime ( (y_M (r))_+ ) \bar z(r) \downarrow dW(r).
\end{align}
Next, from  \eqref{ab:912}, using that $|\varphi_M(r)| \le M$ and
Doob's inequality, we have \be \label{ab919} \E \sup_{ t \in [0,T]}
y_M^2(t) \leq 2M^2S\E \int_0^S \bar z^2(r) dr + 8\E \int_0^S \bar
z^2(r) dr \equiv C_1  < \infty . \ee Let
\[
\tau_{M,n} = \inf \xbr{ t \in [0,S]: y_M(t) \geq n  }, \quad n \in \N,
\]
where infimum over an empty set, by convention, is taken to be $S$.  Then
\begin{align} \label{ab922}
n \E \int_0^t\ONE_{(n,\infty)}((y_M(r))_+) |\bar z(r)|  dr
&\le n\E \ONE_{\{\tau_{M,n}< S\}}\int_{\tau_{M,n}\wedge t}^t |\bar z(r)|  dr\\ \notag
&\le n \left ( \Pp(\sup_{ t \in [0,S]} y_M(t) \ge n)\right)^{1/2} \left (\E \left(\int_{\tau_{M,n}\wedge t}^t \bar z(r) dr\right)^2\right)^{1/2} \\ \notag
&\le C_1^{1/2}\left (\E \left [(t-\tau_{M,n}\wedge t)\int_{0}^S \bar z^2(r) dr\right]\right)^{1/2},
\end{align}
where the third inequality is a consequence of \eqref{ab919}.
Since $(t-\tau_{M,n}\wedge t)$ converges to $0$ as $n \to \infty$ and $\E\int_{0}^S \bar z^2(r) dr < \infty$, we have
that the expression on the last line of the above display converges to $0$ as $n\to \infty$.  Thus we have shown that
\be
\label{ab:938}
\lim_{n\to \infty} n \E \int_0^t\ONE_{(n,\infty)}((y_M(r))_+) |\bar z(r)|  dr = 0 .
\ee
Taking expectations in \eqref{ab922a} and noting that since $\xi_n^\prime$ is bounded, the expectation of the third term on the right side of \eqref{ab922a} is zero, we have
$$
\limsup_{n\to \infty} \E \xi_{n} ( y_M (t)_+ ) \le \frac{M^2}{\alpha} \int_0^t \E ( y_M (s)_+ )^2 ds.
$$
Finally, noting that $\xi_n(u) \to u^2$ as $n \to \infty$, for all $u \in [0, \infty)$, we have by Fatou's lemma that
$$
 \E ( y_M (t)_+ )^2 \le \frac{M^2}{\alpha} \int_0^t \E ( y_M (s)_+ )^2 ds.
$$
Gronwall's lemma now yields that $(y_M(t))_+ = 0$ for all $t \in [0, S]$. A similar argument shows that
$(y_M(t))_-$ and consequently \eqref{eq:ab800} follows.
As argued earlier, this proves the desired uniqueness.
 \end{proof}

\section{Proof of Theorem~\ref{thm: asympink} } \label{sec: PfThm3}
Fix $0 \le t \le S \le T$.  The representation in Lemma~\ref{lemma: ExistenceHopfCole} and \eqref{eqn: InverseHopfCole} give
\begin{align} \notag
y_S^k(t,x) &= - \log \E \xcmd{  e^{-h_0( X^S (0,x) ) } \exp\{ - Z^k (t,x) -\frac{1}{2}C^k(0)t \} }{\Fc^S_t}  \\
& = - \log \E \xcmd{ \E \xcmd{ e^{-h_0( X^S (0,x) ) } \exp\{ - Z^k (t,x)-\frac{1}{2}C^k(0)t \} } { \Fc^S_t\vee \sigma\{W(t)\} } } {\Fc^S_t}. 
\label{eqn: ConditionalExpectation} 
\end{align}
Define a $C([0,t]:\R)$ valued random variable $X^{S,t}$ as
$$X^{S,t}(r) = X^S(r,x), \; r \in [0, t]$$
and a $C([0,S]: \R^{\infty})$ valued random variable $\beta$ as
$$ \beta(r) = (\beta_m(r))_{m\ge 1}, \; r \in [0, S].$$
Then there is a measurable map
$$\Psi: C([0,t]:\R) \times C([0,S]: \R^{\infty}) \to \R_+$$
such that
$$\Psi(X^{S,t},\beta) = \exp\{ - Z^k (t,x) -\frac{1}{2}C^k(0)t \}.$$
In fact one has the following characterization of $\Psi$.  For $\omega \in C([0,t]: \R)$ define
$$M_{\omega}^k(t) = \sum_{m \in \N} \int_0^t \langle \zz^k_{\omega(r)}, \gamma_m \rangle d\beta_m(r).$$
Then $\Psi$ satisfies
$$\Psi(\omega, \beta) = \exp \{ -M_{\omega}^k(t) - \frac{1}{2}C^k(0) t \}, \mbox{ for all } \omega \in C([0,t]: \R), \; a.s.$$
Let $\Pp^{\mu,\nu}_t$ denote the Brownian bridge measure on $C([0,t]: \R)$ with starting point $\mu$ and ending point $\nu$.  Define
$\Psi_0: [0,S]\times \R \times \R \times C([0,S]: \R^{\infty}) \to \R_+$ as
$$
\Psi_0(t,\mu,\nu, \vartheta) = \int_{C([0,t]:\R)}\Psi(\omega,\vartheta) d\Pp^{\mu,\nu}_t(\omega).$$
In particular
\begin{align*}
\Psi_0(t,\mu,\nu, \beta)	=& \int_{C([0,t]:\R)}\Psi(\omega,\beta) d\Pp^{\mu,\nu}_t(\omega)\\
=& \int_{C([0,t]:\R)} \exp \{ -M_{\omega}^k(t) - \frac{1}{2}C^k(0) t \} d\Pp^{\mu,\nu}_t(\omega)\\
\equiv & \E^{\mu,\nu}_t \left [\exp \{ -M_{\bullet}^k(t) - \frac{1}{2}C^k(0) t \} \right].
\end{align*}
Next, using the independence of $W(t)$ and $\Fc^S_t$ we have
\begin{align*}
\E \xcmd{ \exp\{ - Z^k (t,x)-\frac{1}{2}C^k(0)t \} } { \Fc^S_t\vee \sigma\{W(t)\}}	
	=& \E \xcmd{ \Psi(X^{S,t},\beta) } { \Fc^S_t\vee \sigma\{W(t)\}}\\
	=& \Psi_0(t,\gamma + W_t , \gamma, \beta)	\\
	=& \E^{\gamma + W_t,\gamma}_t \left [\exp \{ -M_{\bullet}^k(t) - \frac{1}{2}C^k(0) t \} \right],
\end{align*}
where $\gamma = x+ W(S) - W(t)$.
Therefore
\begin{align*}
&\E \xcmd{  e^{-h_0( X^S (0,x) ) } \exp\{ - Z^k (t,x) -\frac{1}{2}C^k(0)t \} }{\Fc^S_t}\\
=& \E  	\xcmd{e^{-h_0( \gamma + W(t) ) }\E^{\gamma + W(t),\gamma}_t \left [\exp \{ -M_{\bullet}^k(t) - \frac{1}{2}C^k(0) t \} \right]}{\Fc^S_t} \\
=& \int_{\R} e^{-h(y)}G_t(\gamma - y) \E^{y,\gamma}_t \left [\exp \{ -M_{\bullet}^k(t) - \frac{1}{2}C^k(0) t \} \right] dy,
\end{align*}
where $G_t$ is the standard Heat Kernel.  The last expression is seen from  expression (2.17) of~\cite{BertiniCancrini} to be same as $\psi^k_t(\gamma)$,
where $\psi^k_t$ is the solution of the regularized stochastic heat equation.
$$\psi^k_t(x) = G_t\star \psi_0(x) + \int_0^t \langle G_{t-s}\star \psi_s^{k}, dB_s^k \rangle  .$$
(See Section 2.2 of \cite{BertiniCancrini}.)
Therefore 
$$y^k_S(t,x) = -\log \psi^k_t(x+ W(S)- W(t)), \; 0 \le t \le S \le T.$$
In particular
$$y^k_S(S,x) = -\log \psi^k_S(x), \; 0  \le S \le T.$$
The result now follows from Theorem 2.2 of \cite{BertiniCancrini}. 
$\Box$

\section{Appendix.} \label{sec: StochCalculus}
In this section we collect some basic results on forward-backward stochastic
integrals that are used at various places in this work.  Most of the
statements follow by minor modifications of classical results (eg.
\cite{PengPardoux}) and thus only partial sketches are provided.
Throughout this section we will fix $S \in (0, \infty)$, $x \in \R$ and $k \in \mathbb{N}$. As previously, we will suppress
$k$ and $z$ from the notation when writing $Z^k(t,x)$, $\tilde Z^k(t,x)$ etc.

Define $\sigma$-fields
$$
\clg_r^S = \clf_{r,S}^W \vee \clf_S^B, \; \clh_r^S = \clf_{0,S}^W \vee \clf_r^B, \; \tilde \clg_r^S =  \clf_S^{\tilde W} \vee \clf_{r,S}^{\tilde B},
\; \tilde \clh_r^S = \clf_r^{\tilde W} \vee \clf_{0,S}^{\tilde B}.$$
Abusing terminology, we say a stochastic process $\{A(r)\}_{0\le r \le S}$ is adapted to a collection of $\sigma$-fields $\{\mathcal{U}_r\}_{0\le r \le S}$
if $A(r)$ is $\mathcal{U}_r$ measurable for every $r \in [0, S]$.
For such a family of $\sigma$-fields we denote by $\cla^2(\mathcal{U})$ the collection of all adapted processes $\{A(r)\}$ such that
$\int_0^S |A(r)|^2 dr < \infty$ a.s.  Then the following stochastic integrals are well defined:
\begin{align*}
	\int_0^t A(r) \downarrow dW(r) , \; A \in \cla^2(\clg^S); \;& \int_0^t A(r)  dZ(r) , \; A \in \cla^2(\clh^S), \\
	\int_0^t A(r) \downarrow d\tilde Z(r) , \; A \in \cla^2(\tilde \clg^S); \;& \int_0^t A(r)  d\tilde W(r) , \; A \in \cla^2(\tilde \clh^S), t \in [0, S].
\end{align*}
Indeed, consider for example the first stochastic integral.  If $A$ is of the form  $A(r) = \zeta
\ONE_{[a,b)}(r)$, where $\zeta$ is a bounded $\clg^S_b$ measurable random variable
and $0\le a<b\le S$, then
$$\int_0^t A(r) \downarrow dW(r) \equiv \zeta \left (W(b\wedge t) - W(a\wedge t)\right ).$$
The integral is extended to linear combinations of such elementary processes by linearity, and then by denseness and $L^2$-isometry to
all $A \in \cla^2(\clg^S)$ satisfying $\E\int_0^S |A(r)|^2 dr < \infty$; and finally by localization to all $A \in \cla^2(\clg^S)$.


The following elementary lemma gives a basic relation between forward and backward integrals.
\begin{lemma} \label{lemma: FBRelationAux}
Let $K \in \cla^2(\clh^S)$ and $H \in \cla^2(\tilde \clh^S)$.
Let
$$
\tilde K(t) = K(S-t), \;\; \tilde H(t) = H(S-t), \;\; t \in [0, S].$$
Then
$\tilde K \in \cla^2(\tilde \clg^S)$ and $\tilde H \in \cla^2(\clg^S)$.
Furthermore, for $t \in [0, S]$,
\begin{align*}
\int_0^t H(r) d\tilde W(r) =& - \int_{ S-t }^S \tilde H( r ) \downarrow d W(r),\\
\int_0^t K(r) d Z(r) =& - \int_{ S-t }^S \tilde K(  r ) \downarrow d\tilde Z(r)
\end{align*}
\end{lemma}
\begin{proof}
The first statement in the lemma is an immediate consequence of \eqref{eq:1233a} and \eqref{eq:1233}.
  Of the two equalities in the above display, we only prove the first one. The proof of the second identity follows by a similar argument. Consider first the case where $H(t) =  \zeta \ONE_{(a,b]}(t)$, where $\zeta$ is a  bounded $\tilde \clh^S_a$ measurable random variable, and $0\le a<b\le S$.  In that case, note that 
$$
\tilde H(r) = H(S-r)= \zeta \ONE_{ [a,b) } (S-r) = \zeta \ONE_{ [S-b,S-a) } (r),$$ and,  
\begin{align*}
\int_{ S-t }^S \tilde H&(  r ) \downarrow d W(r) \\
&= \zeta \left[ \parbar{    W( ( S-a) \vee (S-t) )   }   - \parbar{    W( ( S-b) \vee (S-t) )    }   \right] \\
&= \zeta \left[ \parbar{    W( ( S-a) \vee (S-t) )  -  W(S)  }   - \parbar{    W( ( S-b) \vee (S-t) )  - W(S)  }   \right] \\
&= \zeta \parbar { \tilde W(t\wedge a) - \tilde W(t\wedge b)} \\
&= -\int_0^t H(r)d\tilde W(r).
\end{align*}
The general case follows by linearity, denseness (along with $L^2$ isometry) and a localization argument.  Details are omitted.
\end{proof}
As an immediate consequence of the lemma we have the following corollary.
\begin{corollary} \label{lemma: FBRelation}
 A pair of processes $(\hat u, \hat v)  \in
\mathscr{H}_S^\infty ( \tilde \clf^S)   \times \mathscr{H}_S^2 (
\tilde \clf^S )$  solves
\eqref{eqn: TimeReversalLinear} if and only if $(u,v)$, defined as $(u(t), v(t)) = (u(S-t), v(S-t))$, $t \in [0, S]$, solves \eqref{eqn: HopfCole}.
\end{corollary}
\begin{proof}
The proof is immediate from Lemma \ref{lemma: FBRelationAux}.
\end{proof}
The following elementary lemma will be  used in the proof of \eqref{ab850}.
\begin{lemma} \label{def: BackwardAdapted}
	Let $\varphi$ be a $C^1$ function on $\R$ and $\psi: [0,S] \to \R$ be a continuous function.
	Suppose that for all $z, z' \in \R$, $\varphi(z) = \varphi_1(z-z')\varphi_2(z')$ for some continuous functions $\varphi_1, \varphi_2$.
	Then for all $t \in [0,S]$,
	\begin{align} \notag \int_0^t \varphi(\tilde Z(r)) \psi(r) d\tilde Z(r) =& \varphi_2(\tilde Z(T))\int_0^t \varphi_1(\tilde Z(r) - \tilde Z(T)) \psi(r) \downarrow d\tilde Z(r) \\
		-& C^k(0) \int_0^t \varphi'(\tilde Z(r)) \psi(r) dr.\label{ab648}\end{align}
\end{lemma}
\begin{proof}
Fix $t \in [0, S]$ and let $\Pi_n = \{0 = t_0^{(n)} < t_1^{(n)} < t_2^{(n)} \cdots < t_k^{(n)} = t\}$ be a partition of $[0,t]$ such that $|\Pi_n| \to 0$ as $n \to \infty$. 
Then (suppressing $n$) letting $\Delta_i \tilde Z = \tilde Z(t_{i+1} ) - \tilde Z(t_{i})$, we see that
%
$\varphi_2(\tilde Z(T))\int_0^t \varphi_1(\tilde Z(r) - \tilde Z(T)) \psi(r) \downarrow d\tilde Z(r)$ is the limit in probability of
\begin{align} \notag
&\varphi_2(\tilde Z(T)) \sum_{i=0}^{k-1} \varphi_1(\tilde Z(t_{i+1})-\tilde Z(T)) \psi(t_{i+1})  \Delta_i \tilde Z \\
&= \sum_{i=0}^{k-1} \varphi(\tilde Z(t_{i+1})) \psi(t_{i+1})  \Delta_i \tilde Z \nonumber\\
& =  \sum_{i=0}^{k-1} \varphi(\tilde Z(t_{i})) \psi(t_{i+1}) \Delta_i \tilde Z \nonumber\\
 & \quad + \sum_{i=0}^{k-1} \parbar{ \varphi(\tilde Z(t_{i+1})) - \varphi(\tilde Z(t_{i})) } \psi(t_{i+1}) \Delta_i \tilde Z. \label{eqn: sumproof}
 \end{align}
 From standard arguments it follows that, in probability, 
 \[
 \lim_{n\to \infty } \sum_{i=0}^{k-1} \parbar{ \varphi(\tilde Z(t_{i+1})) - \varphi(\tilde Z(t_{i})) } \psi(t_{i+1}) \Delta_i \tilde Z = C^k(0) \int_0^t \varphi'(\tilde Z(r)) \psi(r) dr.
 \]
 Likewise, it is easily seen that 
 \[
\lim_{n\to \infty } \sum_{i=0}^{k-1} \varphi(\tilde Z(t_{i})) \psi(t_{i+1}) \Delta_i \tilde Z=\int_0^t \varphi(\tilde Z(r)) \psi(r) d\tilde Z(r),
 \]
 in probability. These two identities  combined with~\eqref{eqn: sumproof} give the result.
\end{proof}
We now present a variation of It\^{o}'s formula that is used in our work.
\begin{lemma} \label{lemma: Ito}
Let $\phi \in C^2 (\R)$.\\
(i) Let  processes $\alpha \in \Hspace{\infty}{\Fc^S}, \beta, \gamma , \delta \in \Hspace{2}{\Fc^S}$ be such that
\[
\alpha(t)= \alpha(0) + \int_0^t \beta(r) dr + \int_0^t \gamma(r) d Z^k (r) + \int_0^t \delta(r) \downarrow dW(r), \quad 0 \leq t \leq T.
\]
Then, for all $t \in [0, S]$
\begin{align*}
\phi( \alpha(t) ) &= \phi( \alpha(0 ) ) + \int_0^t \phi^\prime (\alpha(r) ) \beta(r) dr + \int_0^t \phi^\prime (\alpha(r) ) \gamma(r) dZ^k(r) \\
& \quad + \int_0^t \phi^\prime (\alpha(r) ) \delta(r) \downarrow dW(r)  + \frac{C^k (0 )}{2}  \int_0^t \phi^{\prime \prime} (\alpha(r) ) \gamma(r)^2dr\\
& \quad - \frac{1}{2} \int_0^t \phi^{\prime \prime} (\alpha(r) ) \delta(r)^2 dr.
\end{align*}
(ii) Let  processes $\alpha \in \Hspace{\infty}{\tilde \Fc^S}, \beta, \gamma , \delta \in \Hspace{2}{\tilde \Fc^S}$ be such that
\[
\alpha(t)= \alpha(0) + \int_0^t \beta(r) dr + \int_0^t \gamma(r) \downarrow d \tilde Z^k (r) + \int_0^t \delta(r)  d\tilde W(r), \quad 0 \leq t \leq T.
\]
Then, for all $t \in [0, S]$
\begin{align*}
\phi( \alpha(t) ) &= \phi( \alpha(0 ) ) + \int_0^t \phi^\prime (\alpha(r) ) \beta(r) dr + \int_0^t \phi^\prime (\alpha(r) ) \gamma(r) \downarrow d\tilde Z^k(r) \\
& \quad + \int_0^t \phi^\prime (\alpha(r) ) \delta(r)  d\tilde W(r)  - \frac{C^k (0 )}{2}  \int_0^t \phi^{\prime \prime} (\alpha(r) ) \gamma(r)^2dr\\
& \quad + \frac{1}{2} \int_0^t \phi^{\prime \prime} (\alpha(r) ) \delta(r)^2 dr.
\end{align*}
\end{lemma}
\begin{proof}
	We will only consider (i).  The statement in (ii) follows similarly.  The proof follows using standard arguments (cf. Theorem 3.3.3 in \cite{KaratzasBook}).  We merely comment on one key point.  Suppose that $\phi''$ is bounded. (The general case can be reduced to such a setting by localization.)
Fix $t \in [0, S]$ and let $\Pi_n = \{0 = t_0^{(n)} < t_1^{(n)} < t_2^{(n)} \cdots < t_k^{(n)} = t\}$ be a partition of $[0,t]$ such that $|\Pi_n| \to 0$ as $n \to \infty$.  Then the only change to standard proofs is in the treatment
of the term
\be
\label{ab807}
\sum_{i=1}^{k} \phi''(\alpha_{i-1})\Delta_iW \Delta_iZ ,
\ee
where for a process $\zeta$, we write $\Delta_i \zeta = \zeta(t_{i} ) - \zeta( t_{i-1})$. 
One needs to argue that the expression in \eqref{ab807} approaches $0$ as $n \to \infty$, which follows on noting that
\begin{align*}
	\E \left [ \sum_{i=1}^{k} \phi''(\alpha_{i-1})\Delta_iW \Delta_iZ\right]^2 =& \E \left [ \sum_{i=1}^{k} \left(\phi''(\alpha_{i-1})\right)^2 (\Delta_iW)^2 (\Delta_iZ)^2\right]\\
	\le & \sup_x|\phi''(x)|^2 \E \left [ \sum_{i=1}^{k}  (\Delta_iW)^2 \E \xcmd{(\Delta_iZ)^2}{\clf_{t_{i-1}}^B\vee\clf_S^{W}}\right]\\
	= & \sup_x|\phi''(x)|^2C^k(0) \sum_{i=1}^{k} (t_i-t_{i-1})\E(\Delta_iW)^2\\
	=& \sup_x|\phi''(x)|^2C^k(0) \sum_{i=1}^{k} (t_i-t_{i-1})^2,
\end{align*}
where the first equality follows on noting that by a conditioning argument the cross-product terms do not contribute while the next to last equality follows from \eqref{ab824}.
\end{proof}
Finally, we give the proof of \eqref{ab850}.\\ 

\noindent {\bf Proof of \eqref{ab850}.}
Note that $\{\tilde Z(t)\}_{t\in [0, S]}$ is a martingale with respect to the filtration $\tilde \clg_t = \clf_{0,t}^{\tilde B} \vee \clf_S^{\tilde W}$, with quadratic 
variation given as $\langle Z\rangle_t = C^k(0) t$. Thus, by an application of It\^{o}'s formula, we have that
\begin{align}
E(S) - E(t) &= -\left [ \int_t^S E(r)d\tilde Z (r) - C^k (0) \int_t^S E(r)dr \right ]  \nonumber\\
& =  - E(T)\int_t^S \frac{E(r)}{E(T)} \downarrow d\tilde Z (r), \label{eq:1334}
\end{align}
where the second equality is a consequence of Lemma \ref{def: BackwardAdapted} on taking $\varphi(x) = \varphi_1(x) = \varphi_2(x) = e^{-x}$ and $\psi (t) = \exp\{\frac{1}{2}C^k(0)t\}$.
Also, recall from \eqref{eqn: Jdef}
that
\begin{equation}M(t) = M(S) - \int_t^S J(r) d\tilde W(r), \; 0 \le t \le S. \label{eq:1336}
\end{equation}
Let $\Pi_n = \{t = t_0^{(n)} < t_1^{(n)} < t_2^{(n)} \cdots < t_k^{(n)} = S\}$ be a partition of $[t,S]$ such that $|\Pi_n| \to 0$ as $n \to \infty$.
Then, suppressing $n$ in the notation
\begin{align*}
	U(t) - U(S) =&  -\sum_{i=1}^k (U(t_i) - U(t_{i-1})) \\
	=& -\sum_{i=1}^k (M(t_i)E(t_i) - M(t_{i-1})E(t_{i-1}))\\
	=& -\sum_{i=1}^k M(t_i)(E(t_i) -E(t_{i-1})) - \sum_{i=1}^k E(t_{i-1}) (M(t_i) - M(t_{i-1}))\\
	\end{align*}										
The equality in \eqref{ab850} now follows from \eqref{eq:1334} and \eqref{eq:1336} on taking limit as $n\to \infty$ in the last line.

\def\cprime{$'$} \def\cprime{$'$}


\end{document}